\RequirePackage{fix-cm}
\documentclass[11pt,a4paper,oneside,british]{amsart}
\usepackage{charter}
\usepackage[type1]{biolinum}
\usepackage[T1]{fontenc}
\usepackage[latin9]{inputenc}
\usepackage{babel}
\usepackage{prettyref}
\usepackage{mathtools}
\usepackage{enumitem}
\usepackage{amsthm}
\usepackage{amssymb}
\usepackage{setspace}
\usepackage{xargs}[2008/03/08]
\usepackage[unicode=true,
 bookmarks=true,bookmarksnumbered=true,bookmarksopen=false,
 breaklinks=false,pdfborder={0 0 1},backref=false,colorlinks=false]
 {hyperref}
\hypersetup{pdftitle={Varieties with Ample Tangent Sheaves},
 pdfauthor={Philip Sieder}}

\makeatletter

\pdfpageheight\paperheight
\pdfpagewidth\paperwidth

\newcounter{subbsection}[subsection]
\renewcommand*{\thesubbsection}{%
  \ifnum\value{subsection}=0 %
    \thesection
  \else
    \thesubsection
  \fi
}
 \newcommand\thmsname{\protect\theoremname}
 \newcommand\nm@thmtype{theorem}
 \theoremstyle{plain}
 
 \newenvironment{namedthm}[1][Undefined Theorem Name]{
   \ifx{#1}{Undefined Theorem Name}\renewcommand\nm@thmtype{theorem*}
   \else\renewcommand\thmsname{#1}\renewcommand\nm@thmtype{namedtheorem}
   \fi
   \begin{\nm@thmtype}}
   {\end{\nm@thmtype}}
  \theoremstyle{plain}
  \newtheorem*{thm*}{\protect\theoremname}
 \theoremstyle{plain}
 \newtheorem{thm}{Theorem}[subbsection]
  \theoremstyle{definition}
  \newtheorem{defn}[thm]{\protect\definitionname}
  \theoremstyle{remark}
  \newtheorem*{rem*}{\protect\remarkname}
  \theoremstyle{plain}
  \newtheorem{prop}[thm]{\protect\propositionname}
  \theoremstyle{plain}
  \newtheorem{lem}[thm]{\protect\lemmaname}
  \theoremstyle{remark}
  \newtheorem*{acknowledgement*}{\protect\acknowledgementname}

\newcommand*{\longhookrightarrow}{\ensuremath{\lhook\joinrel\relbar\joinrel\rightarrow}}

\DeclareMathOperator{\sHom}{{\mathcal H}\hspace{-.1em}\mathit{om}}

\global\long\def\ox#1{\mathcal{O}_{#1}}
\global\long\def\tx#1{\mathcal{T}_{#1}}
\newcommandx\sing[1][usedefault, addprefix=\global, 1=X]{\operatorname{Sing}(#1)}
\newcommandx\pn[1][usedefault, addprefix=\global, 1=n]{\mathbb{P}_{#1}}
\newcommandx\cotang[2][usedefault, addprefix=\global, 1=X, 2=1]{\Omega_{#1}^{#2}}
\newcommandx\Hom[2][usedefault, addprefix=\global, 1=\cotang, 2=\ox X]{\sHom(#1,#2)}
\global\long\def\complex{\mathbb{C}}

\usepackage{faktor,tikz-cd}

\makeatother

  \providecommand{\acknowledgementname}{Acknowledgement}
  \providecommand{\definitionname}{Definition}
  \providecommand{\lemmaname}{Lemma}
  \providecommand{\propositionname}{Proposition}
  \providecommand{\remarkname}{Remark}
  \providecommand{\theoremname}{Theorem}
\providecommand{\theoremname}{Theorem}

\begin{document}

\title{Varieties with Ample Tangent Sheaves}

\author{Philip Sieder}

\begin{abstract}
	This paper generalises Mori's famous theorem about ``Projective manifolds
	with ample tangent bundles'' \cite{Mori79} to normal projective
	varieties in the following way:
	
	A normal projective variety over $\complex$ with ample tangent sheaf
	is isomorphic to the complex projective space.
\end{abstract}

\maketitle
\thispagestyle{empty}

\section{Introduction}

In this paper we give a proof for the following theorem.
\begin{namedthm}[Main Theorem]
A normal projective variety over $\complex$ with ample tangent sheaf
is isomorphic to the projective space.
\end{namedthm}
We work over the field of complex numbers $\complex$. Besides that
restriction, the theorem is a generalisation to singular varieties
of Mori's famous result.
\begin{thm*}[\cite{Mori79}]
An $n$-dimensional projective manifold $X$ over an algebraically
closed field $\mathbb{K}$ with ample tangent bundle is isomorphic
to the projective space $\mathbb{P}_{\mathbb{K}}^{n}$.
\end{thm*}
Mori's work has been generalised over the years in various ways, for
example by Andreatta and Wi{\'s}niewski \cite{AndeattaWisniewski}:
For $X$ being $\pn$ it suffices that $\tx X$ contains an ample
subbundle. This has been altered by Aprodu, Kebekus and
Peternell \cite[Section 4]{GaloisCoverings}. They add the assumption
that $X$ has Picard number 1, but an ample subsheaf (not necessarily
locally free) of $\tx X$ then induces $X\simeq\pn[n]$. 
Generalising those results, Liu \cite{Liu2016AKP} recently showed that $X$ is already the projective space if $\tx X$ contains an ample subsheaf (again not necessarily locally free).
%
Kebekus \cite{Kebekus01}
even characterises $\pn[n]$ only by using the anticanonical degree
of all rational curves being greater than $n$. All these efforts, besides
Ballico's article \cite{Ballico}, keep the preliminary that $X$
is smooth. Ballico's paper on the other hand treats mainly positive
characteristic, as he requires the tangent sheaf to be locally free.
Which, the Zariski-Lipman conjecture suggests, is most likely never
the case over the complex numbers, if $X$ is singular.

\subsubsection*{Outline of our proof}

We consider a special desingularisation $\hat{X}$ of the given variety
$X$ of dimension $\geq2$ (normal curves are smooth) and prove that
$\hat{X}$ is the projective space. As $\pn[n]$ is minimal, $X$
itself is already the projective space. To show that $\hat{X}$ is
the projective space, we combine two strong results. 

First, we relate $\tx X$ to $\tx{\hat{X}}$: For a suitable desingularisation
$\pi \colon \hat{X}\rightarrow X$, there is a morphism $f \colon \pi^{*}\tx X\rightarrow\tx{\hat{X}}$
that is an isomorphism outside $\pi^{-1}(\sing[X])$ (\prettyref{thm:vectorfieldsGeneral}). 

Secondly, we use a corollary given by Cho, Miyaoka and Shepherd-Barron
\cite[Corollary 0.4 (11)]{ChoMijaokaShepherd} that Kebekus \cite{Kebekus01}
later proved directly (although he claims a weaker result): A uniruled
manifold $\hat{X}$ is isomorphic to the projective space, if the
anticanonical degree $-K_{\hat{X}}.\hat{C}$ is greater or equal $n+1$
for all rational curves $\hat{C}$ through a general point $p$. The
uniruledness of $\hat{X}$ follows from the negativity of $K_{\hat{X}}$
and the anticanonical degree is calculated using the splitting of
$\tx{\hat{X}}|_{\hat{C}}$ on the normalisation of $\hat{C}$ (\prettyref{lem:TxCurves}).
Hence $\hat{X}\simeq\pn[n]\simeq X$.

\section{Preliminaries\label{sec:Preliminaries}}

Let us first recall the definition of the tangent sheaf for a proper
variety, as it is a central term in this paper.
\begin{defn}[tangent sheaf]
Let $X$ be a algebraic variety, then its \emph{tangent sheaf} $\tx X\coloneqq\Hom$
is the dual of the cotangent sheaf.

We want to work on a desingularisation $\hat{X}$ of the normal variety
$X$, so we have to connect $\tx X$ with $\tx{\hat{X}}$:
\end{defn}
\begin{thm}
\label{thm:isomorphismTxTdesing}Let $X$ be a normal projective variety
with tangent sheaf $\tx X$. Then there is a desingularisation $\pi \colon \hat{X}\rightarrow X$
and an $\ox X$-module isomorphism 
\[
\tx X\rightarrow\pi_{*}\tx{\hat{X}}.
\]
\end{thm}
\begin{proof}
Graf and Kov\'{a}cs \cite[Theorem 4.2]{GrafKovacsOptimalExtension}
state that there is a resolution $\pi \colon \hat{X}\rightarrow X$ such
that $\pi_{*}\tx{\hat{X}}$ is reflexive. The sheaves $\tx X$ and
$\pi_{*}\tx{\hat{X}}$ are reflexive, $X$ is normal and $\pi$ is
an isomorphism outside the preimage of a set of codimension $2$.
Thus we obtain an isomorphism $\tx X\rightarrow\pi_{*}\tx{\hat{X}}$.
\end{proof}
\begin{rem*}
For a more thorough understanding of the map $\tx X\rightarrow\pi_{*}\tx{\hat{X}}$
and the resolution $\pi$, see the paper of Greb, Kebekus and Kov\'{a}cs
\cite[Section 4]{GKKExtensionDifferentialForms}.
\end{rem*}
The most cited definition for ample sheaves is in Ancona's paper \cite{AnconaAmple}.
He defines ampleness and provides some equivalent characterisations,
but gives very few properties. Kubota \cite{KubotaAmple} on the other
hand works over graded $\ox X$-modules and gives some properties,
but does not use the most modern language.

So we recall a definition and the most important properties we use
throughout this work.
\begin{defn}[ample sheaf]
 Let $X$ be a proper algebraic variety and $\mathcal{E}$ a coherent
sheaf on $X$. Then we say $\mathcal{E}$ is \emph{ample} if for every
coherent sheaf $\mathcal{F}$ on $X$ there exists an $n=n(\mathcal{F})$ such that $\mathcal{F}\otimes S^{m}\mathcal{E}$
is globally generated for $m\geq n$.
\end{defn}
\begin{rem*}
Other characterisations of ampleness can be found in \cite{AnconaAmple}.
Note that an ample sheaf, unlike an ample vector bundle, on a proper
variety $X$ does not yield that its support is projective, but only
Moishezon \cite[Remark p.\ 244]{GrauertPeternellRemmert1994}.
\end{rem*}
The following properties can be found in Debarre's paper \cite[Section 2]{Debarre}
or the proof in the vector bundle case (as in \cite{LazarsfeldPositivityII})
carries over to coherent sheaves:
\begin{prop}
Let $X$ and $Y$ be normal projective varieties, $f \colon Y\rightarrow X$
a finite morphism, $\mathcal{E}$, $\mathcal{E}_{1}$ and $\mathcal{E}_{2}$
 sheaves of $\ox X$-modules and $\mathcal{E}$ ample, then 
\vspace{-\parskip}
\begin{enumerate}
\item $f^{*}\mathcal{E}$ is ample (in particular restrictions of ample
sheaves are ample)
\item every quotient of $\mathcal{E}$ is ample
\item $\mathcal{E}_{1}\oplus\mathcal{E}_{2}$ is ample if and only if $\mathcal{E}_{1}$
and $\mathcal{E}_{2}$ are both ample
\end{enumerate}
\end{prop}
\begin{prop}[{\cite[6.4.17]{LazarsfeldPositivityII}}]
\label{prop:curvemapAmple}Let $C$ be a smooth curve and ${\mathcal{E}}$
and ${\mathcal{F}}$ vector bundles on $C$. If ${\mathcal{E}}$ is
ample and there is a homomorphism ${\mathcal{E}}\rightarrow{\mathcal{F}}$,
surjective outside of finitely many points, then ${\mathcal{F}}$
is ample.
\end{prop}
We need one further result which is, besides \prettyref{thm:isomorphismTxTdesing},
the main ingredient for our result:
\begin{thm}[{\cite[Corollary 0.4 (11)]{ChoMijaokaShepherd}}]
\label{thm:choMijaoka}A uniruled projective complex manifold $X$
of dimension $n$ with a dense open subspace $U$ such that for all
$p\in U$ and all rational curves $C$ through $p$ the inequality
$-K_{X}.C\geq n+1$ holds, is isomorphic to $\pn[n]$.
\end{thm}

\section{\label{sec:dim3}Projective varieties with ample tangent sheaves}

Now we get to the main result of the paper:
\begin{thm}
\label{thm:mainTheorem}Let $X$ be a normal projective variety over
$\complex$ of dimension $n$ with ample tangent sheaf $\tx X$, then
\[
\setlength{\abovedisplayskip}{1\abovedisplayshortskip}X\simeq\pn[n].
\]
\end{thm}
Before proving the main theorem we have to adapt the results given
in \prettyref{sec:Preliminaries}.
\begin{thm}
\label{thm:vectorfieldsGeneral}Let $X$ be a normal projective variety,
then there is a desingularisation $\pi \colon \hat{X}\rightarrow X$ and
an $\ox{\hat{X}}$-module homomorphism
\[
f \colon \pi^{*}\tx X\rightarrow\tx{\hat{X}}
\]
 that is an isomorphism outside $\pi^{-1}(\sing)$.
\end{thm}
\begin{proof}
Using \prettyref{thm:isomorphismTxTdesing}, we obtain an isomorphism
$\tx X\rightarrow\pi_{*}\tx{\hat{X}}$ for a suitable resolution $\pi \colon \hat{X}\rightarrow X$.
The map $\pi$ is an isomorphism outside $\pi^{-1}(\sing[X])$ (one
has to retrace the resolution guaranteed by \cite[Theorem 4.2]{GrafKovacsOptimalExtension} to \cite[Theorem 3.45]{KollarResolutions}
for this property). Pulling back $\tx X\rightarrow\pi_{*}\tx{\hat{X}}$
and using the natural morphism $c \colon \pi^*\pi_{*}\tx{\hat{X}}\rightarrow \tx{\hat{X}}$, there is the diagram\[ \setlength{\abovedisplayskip}{\abovedisplayshortskip}
\begin{tikzcd}
\pi^{*}\tx X \arrow[r,"g"] \arrow[rr,bend left,"f"]& \pi^* \pi_* \tx{\hat{X}} \arrow[r, "c"] & \tx{\hat{X}}.
\end{tikzcd}\] Considering the maps $g$ and $c$, it is easy to check that they,
and therefore $f$, are isomorphisms outside $\pi^{-1}(\sing[X])$.
\end{proof}
\begin{rem*}
The editor pointed out to the author that Kawamata \cite[p.\ 14]{KawamataEditor} made use of the map $f$ as well.
%
\end{rem*}
\begin{lem}
\label{lem:TxCurves}Let $X$ be a normal projective variety of dimension
$n$ with ample tangent sheaf $\tx X$ and $C\subset X$ a closed
curve that intersects $\sing[X]$ in at most finitely many points.
Let $\pi \colon \hat{X}\rightarrow X$ be a desingularisation as in \prettyref{thm:vectorfieldsGeneral},
$\hat{C}$ the strict transform of $C$ and $\eta \colon \tilde{C}\rightarrow\hat{C}$
the normalisation of $\hat{C}$. Accordingly, there is the following
commutative diagram:
\[ 
\begin{tikzcd}
\tilde{C} \arrow[r,"\eta"] \arrow[rd,"\nu"'] & \hat{C} \arrow[d,hook] \arrow[r] & C \arrow[d,hook]\\
& \hat{X} \arrow[r,"\pi"] & X
\end{tikzcd}\] 

Then $\nu^{*}\tx{\hat{X}}$ is an ample vector bundle and the anticanonical
degree $-K_{\hat{X}}.\hat{C}$ is positive. If $\hat{C}$ is a rational
curve, $-K_{\hat{X}}.\hat{C}\geq n+1$.
\end{lem}
\begin{proof}
The choice of $\pi$ yields the map $f \colon \pi^{*}\tx X\rightarrow\tx{\hat{X}}$.
Pulling back $f$ via $\nu$ and dividing out the kernel gives \[
\overline{\nu^{*}f} \colon \mathcal{A}\longhookrightarrow\nu^{*}\tx{\hat{X}}
\] with $\mathcal{A}\coloneqq\faktor{\nu^{*}\pi^{*}\tx X}{\ker(\nu^{*}f)}$.
The sheaf $\mathcal{A}$ is ample, since $\tx X$ is ample, $\pi\circ\nu$
is finite and quotients of ample sheaves are ample again. Moreover
$\mathcal{A}$ is locally free of rank $n$ because it is a torsion-free
sheaf on a smooth curve, $\pi\circ\nu$ is an isomorphism outside
of finitely many points and $\ker(\nu^{*}f)$ is supported on only
finitely many points. Using Proposition \ref{prop:curvemapAmple},
we deduce that $\nu^{*}\tx{\hat{X}}$ is an ample vector bundle. Because
$-K_{\hat{X}}.\hat{C}=\deg\nu^{*}\tx{\hat{X}}$, the anticanonical
degree is certainly positive. Since $\nu^{*}\tx{\hat{X}}$ splits
on $\pn[1]$ and a direct sum of ample vector bundles is ample only
if all summands are ample, we obtain $\nu^{*}\tx{\hat{X}}\simeq\bigoplus_{i=1}^{n}\ox{\pn[1]}(a_{i})$
with $a_{i}\geq1$ for all $i$. The dual of the homomorphism $\nu^{*}\cotang[\hat{X}][1]\rightarrow\cotang[\tilde{C}][1]$
is a non-trivial map ${\tx{\pn[1]}\simeq\ox{\pn[1]}(2)\rightarrow\nu^{*}\tx{\hat{X}}}$.
Thus $a_{i}\geq2$ for at least one $i$ and we can conclude $-K_{\hat{X}}.\hat{C}=\sum_{i=1}^{n}a_{i}\geq n+1$.
\end{proof}
Now we use \prettyref{lem:TxCurves} to show that the assumptions
of \prettyref{thm:choMijaoka} are fulfilled for $\hat{X}$ and hence
$X$ is isomorphic to $\pn[n]$.
\begin{proof}[Proof of \prettyref{thm:mainTheorem}]
Normal curves are smooth, so we can assume that $n\geq2$. Let $\pi \colon \hat{X}\rightarrow X$
be a desingularisation as in \prettyref{lem:TxCurves} and let $p\in\hat{X}\setminus\pi^{-1}(\sing[X])$
be any general point outside the exceptional locus. \\
Since $\hat{X}$ is projective, there is an irreducible curve $\hat{C}$
through $p$. As $\hat{C}$ is the strict transform of a closed curve
$C\subset X$, $K_{\hat{X}}.\hat{C}<0$ according to \prettyref{lem:TxCurves}.
Therefore $\hat{X}$ is uniruled by \cite[Theorem 1]{Uniruledness}.

Any rational curve $\hat{C}\subset\hat{X}$ containing $p$ projects
to a curve $C$ on $X$. The curve $C$ meets $\sing[X]$ in at most
finitely many points, thus \prettyref{lem:TxCurves} applies and we
have the assumptions of \prettyref{thm:choMijaoka} fulfilled. So
$\hat{X}$ is isomorphic to the projective space $\pn[n]$. Hence
$X\simeq\pn[n]$ too.
\end{proof}
\newpage{}
\begin{acknowledgement*}
I want to thank Prof.\ Dr.\ Thomas Peternell for his guidance and
support and Dr.\ Patrick Graf for his advice in many occasions, proof\-reading
and especially for hinting me towards \cite[Theorem 4.2]{GrafKovacsOptimalExtension}.
In addition I thank Andreas Demleitner and Dr.~Florian Schrack for
proof\-reading, their advice and countless conversations.
\end{acknowledgement*}
\bibliographystyle{amsalpha}
\bibliography{MAMaths}

\end{document}